\documentclass[11pt,reqno]{amsart}
\usepackage{color}

\usepackage{amsmath}
\usepackage{amscd}
\usepackage{amstext}
\usepackage{amsopn}

\usepackage{amsfonts,amscd}
\usepackage{amssymb}
\usepackage{url}
\usepackage{hyperref}

\usepackage[english]{babel}
\usepackage{booktabs}
\usepackage{tikz}

\usepackage{enumitem}

\usepackage{mathtools}

\theoremstyle{plain}
\newtheorem{theorem}                 {Theorem}      [section]

\newtheorem{lemma}        [theorem]  {Lemma}
\newtheorem{proposition}  [theorem]  {Proposition}

\theoremstyle{definition}

\newtheorem{definition}   [theorem]  {Definition}

\newtheorem{example}      [theorem]  {Example}

\newtheorem{remark}       [theorem]  {Remark}

\numberwithin{equation}{section}

\def \cn{{\mathbb C}}

\def \hn{{\mathbb H}}

\def \rn{{\mathbb R}}

\def \zn{{\mathbb Z}}

\def \B{\mathcal B}
\def \E{\mathcal E}

\def \P{\mathcal P}

\def\nab#1#2{\hbox{$\nabla$\kern -.3em\lower 1.0 ex
		\hbox{$#1$}\kern -.1 em {$#2$}}}
\def\hatnab#1#2{\hbox{$\nabla$\kern -.3em\lower 1.0 ex
		\hbox{$#1$}\kern -.1 em {$#2$}}}

\def \Re{\mathfrak R\mathfrak e}

\def\la#1{\mathfrak{#1}}

\def\basis{\mathcal{B}}
\def\sed{\quad\quad}
\def\dtatzero{\restr{\frac{d}{dt}}{t=0}}

\def\restr#1#2{{
  \left.\kern-\nulldelimiterspace 
  #1 
  \vphantom{\big|} 
  \right|_{#2} 
  }}

\def \g{\mathfrak{g}}

\def \k{\mathfrak{k}}

\def \p{\mathfrak{p}}

\DeclareMathOperator{\Ad}{Ad}

\def \GLC#1{\mathbf{GL}_{#1}(\cn)}
\def \glc#1{\mathfrak{gl}_{#1}(\cn)}
\def \GLH#1{\mathbf{GL}_{#1}(\hn)}

\def \SL2{\widetilde{\text{\bf SL}}_{2}(\rn)}

\def \SO#1{\mathbf{SO}(#1)}

\def \U#1{\text{\bf U}(#1)}

\def \SU#1{\text{\bf SU}(#1)}

\def \Sp#1{\text{\bf Sp}(#1)}
\def \sp#1{\mathfrak{sp}(#1)}

\DeclareMathOperator{\Div}{div} 

\DeclareMathOperator{\trace}{trace}

\numberwithin{equation}{section}
\allowdisplaybreaks

\begin{document}



\author{Sigmundur Gudmundsson}
\address{Mathematics, Faculty of Science\\
Lund University\\
Box 118, Lund 221 00\\
Sweden}
\email{Sigmundur.Gudmundsson@math.lu.se}

\author{Adam Lindstr\" om}
\address{Institut f\" ur Mathematik\\ 
Universit\" at Wien\\
Oskar-Morgenstern-Platz 1\\
1090 Wien\\
Austria}
\email{Adam.Lindstroem@univie.ac.at}

\title
[A Unifying Framework]
{A Unifying Framework for Complex-Valued Eigenfunctions via The Cartan Embedding}

\begin{abstract}
In this work we find a unifying scheme for the known explicit complex-valued eigenfunctions on the classical compact Riemannian symmetric spaces.  For this we employ the well-known Cartan embedding for those spaces.  This also leads to the construction of new eigenfunctions on the quaternionic Grassmannians.
\end{abstract}


\maketitle

\section{Introduction}
\label{section-introduction}

A complex-valued function $\phi:(M,g)\to\cn$ on a Riemannian manifold is called an {\it eigenfunction} if it is eigen both with respect to the classical {\it Laplace-Beltrami} operator and the so called {\it conformality} operator.  These objects have turned out be very useful for the construction of explicit complex-valued {\it harmonic morphisms}, $p$-{\it harmonic functions} and even {\it minimal submanifolds} of $(M,g)$ of codimension two. For the original references to these constructions see Theorem 2.6 of \cite{Gud-Sak-1}, Theorem 3.1 in \cite{Gud-Sob-1} and Theorem 1.3 from \cite{Gud-Mun-1}. For the general theory of harmonic morphisms between Riemannian manifolds we refer to the excellent book \cite{Bai-Woo-book} and the regularly updated online bibliography \cite{Gud-bib}.
\medskip

Explicit complex-valued eigenfunctions have been constructed on all the classical compact irreducible {\it Riemannian symmetric spaces} $G/K$, see Table \ref{table-eigenfamilies}.  The main aim of this paper is to provide a unified framework for the construction for the seven last examples in the table.  For this we employ the well-known {\it Cartan embedding} mapping such a symmetric space $G/K$ into its isometry group $G$ as a totally geodesic submanifold.    Our main results can be summarised as follows.

\begin{theorem}
\label{theorem-linear-eigenfunctions}
Let $G/K$ be one of the classical compact irreducible symmetric spaces 4.-10. in Table 1. Denote by $\hat\Phi: G\to G$ the corresponding Cartan map. Then there exists a complex matrix $A \in \cn^{n\times n}$ such that the function $\hat{\phi}_A : G \to \cn$ with 
\begin{equation*}
\hat{\phi}_{A}(p) = \trace(\hat\Phi(p)\cdot A)
\end{equation*}
is a $K$-invariant $(\lambda, \mu)$-eigenfunction on $G$, for suitable $\lambda, \mu \in \cn$. This induces a $(\lambda,\mu)$-eigenfunction $\phi_A$ on the quotient $G/K$ with the same eigenvalues.
\end{theorem} 

As an important example we explicitly investigate the special case of the quaternionic Grassmannians 
$\Sp{m+n}/\Sp{m}\times\Sp{n}$.
As a by-product we yield new eigenfunctions on these spaces. We then conclude with a new general construction method for eigenfunctions from Riemannian manifolds. The reader interested in further details is referred to \cite{Lin-MSc}.

\renewcommand{\arraystretch}{1.5}
\begin{table}[h]\label{table-eigenfamilies}
	\makebox[\textwidth][c]{
		\begin{tabular}{ccccc}
			\midrule
			\midrule
No.&$G/K$	& $\lambda$ & $\mu$ & Eigenfunctions \\
\midrule
\midrule
1.& $\SO n$ & $-\,\frac{(n-1)}2$ & $-\,\frac 12$ & 
see \cite{Gud-Sak-1}\\
\midrule
2. &$\SU n$ & $-\,\frac{n^2-1}n$ & $-\,\frac{n-1}n$ & 
see \cite{Gud-Sob-1} \\
\midrule
3. & $\Sp n$ & $-\,\frac{2n+1}2$ & $-\,\frac 12$ & 
see \cite{Gud-Sak-1} \\
\midrule
4. & $\SU n/\SO n$ & $-\,\frac{2(n^2+n-2)}{n}$& $-\,\frac{4(n-1)}{n}$ & 
see \cite{Gud-Sif-Sob-2} \\
\midrule
5.& $\Sp n/\U n$ & $-\,2(n+1)$ & $-\,2$ & 
see \cite{Gud-Sif-Sob-2} \\
\midrule
6. & $\SO{2n}/\U n$ & $-\,2(n-1)$ & $-1$ & 
see \cite{Gud-Sif-Sob-2} \\
\midrule
7. &$\SU{2n}/\Sp n$ & $-\,\frac{2(2n^2-n-1)}{n}$ & $-\,\frac{2(n-1)}{n}$ & 
see \cite{Gud-Sif-Sob-2} \\
\midrule
8. &$\SO{m+n}/\SO m\times\SO n$ & $-(m+n)$ & $-2$ & 
see \cite{Gha-Gud-4} \\
\midrule
9. &$\U{m+n}/\U m\times\U n$ & $-2(m+n)$ & $-2$ & 
see \cite{Gha-Gud-5} \\
\midrule
10. & $\Sp{m+n}/\Sp m\times\Sp n$ & $-2(m+n)$ & $-1$ & 
see \cite{Gha-Gud-5} \\
\midrule
\midrule
\end{tabular}	
}
\bigskip
\caption{Eigenfamilies on the classical compact irreducible Riemannian symmetric spaces.}
\label{table-eigenfamilies}	
\end{table}
\renewcommand{\arraystretch}{1}

\section{Eigenfunctions and Eigenfamilies}
\label{section-eigenfunctions}

Let $(M,g)$ be an $m$-dimensional Riemannian manifold and $T^{\cn}M$ be the complexification of the tangent bundle $TM$ of $M$. We extend the metric $g$ to a complex bilinear form on $T^{\cn}M$.  Then the gradient $\nabla\phi$ of a complex-valued function $\phi:(M,g)\to\cn$ is a section of $T^{\cn}M$.  In this situation, the well-known complex linear {\it Laplace-Beltrami operator} (alt. {\it tension field}) $\tau$ on $(M,g)$ acts locally on $\phi$ as follows
$$
\tau(\phi)=\Div (\nabla \phi)=\sum_{i,j=1}^m\frac{1}{\sqrt{|g|}} \frac{\partial}{\partial x_j}
\left(g^{ij}\, \sqrt{|g|}\, \frac{\partial \phi}{\partial x_i}\right).
$$
For two complex-valued functions $\phi,\psi:(M,g)\to\cn$ we have the following well-known fundamental relation
\begin{equation*}
\label{equation-basic} \tau(\phi\cdot \psi)=\tau(\phi)\cdot\psi +2\,\kappa(\phi,\psi)+\phi\cdot\tau(\psi),
\end{equation*}
where the complex bilinear {\it conformality operator} $\kappa$ is given by $$\kappa(\phi,\psi)=g(\nabla \phi,\nabla \psi).$$  In local coordinates this satisfies 
$$\kappa(\phi,\psi)=\sum_{i,j=1}^mg^{ij}\cdot\frac{\partial\phi}{\partial x_i}\frac{\partial \psi}{\partial x_j}.$$

\begin{definition}\cite{Gud-Sak-1}
\label{definition-eigenfamily}
Let $(M,g)$ be a Riemannian manifold. Then a complex-valued function $\phi:M\to\cn$ is said to be a {\it $(\lambda,\mu)$-eigenfunction} if it is eigen both with respect to the Laplace-Beltrami operator $\tau$ and the conformality operator $\kappa$ i.e. there exist complex numbers $\lambda,\mu\in\cn$ such that $$\tau(\phi)=\lambda\cdot\phi\ \ \text{and}\ \ \kappa(\phi,\phi)=\mu\cdot \phi^2.$$	
A set $\E =\{\phi_i:M\to\cn\ |\ i\in I\}$ of complex-valued functions is said to be a {\it $(\lambda,\mu)$-eigenfamily} on $M$ if there exist complex numbers $\lambda,\mu\in\cn$ such that for all $\phi,\psi\in\E$ we have 
$$\tau(\phi)=\lambda\cdot\phi\ \ \text{and}\ \ \kappa(\phi,\psi)=\mu\cdot \phi\,\psi.$$ 
\end{definition}
\medskip

For the standard odd-dimensional round spheres we have the following eigenfamilies based on the classical real-valued spherical harmonics.

\begin{example}\cite{Gud-Mun-1}
\label{example-basic-sphere} 
Let $S^{2n-1}$ be the odd-dimensional unit sphere in the standard Euclidean space $\cn^{n}\cong\rn^{2n}$ and define $\phi_1,\dots,\phi_n:S^{2n-1}\to\cn$ by
$$\phi_j:(z_1,\dots,z_{n})\mapsto \frac{z_j}{\sqrt{|z_1|^2+\cdots +|z_n|^2}}.$$  Then the tension field $\tau$ and the conformality operator $\kappa$ on $S^{2n-1}$ satisfy
$$\tau(\phi_j)=-\,(2n-1)\cdot\phi_j\ \ \text{and}\ \ \kappa(\phi_j,\phi_k)=-\,1\cdot \phi_j\cdot\phi_k.$$
\end{example}

For the standard complex projective space $\cn P^n$ we have a complex multi-dimensional eigenfamily, obtained in a similar way.

\begin{example}\cite{Gud-Mun-1}
\label{example-basic-projective-space}
Let $\cn P^n$ be the standard $n$-dimensional complex projective space. For a fixed integer $1\le\alpha < n+1$ and some $1\le j\le\alpha < k\le n+1$  define the function $\phi_{jk}:\cn P^n\to\cn$ by
$$\phi_{jk}:[z_1,\dots,z_{n+1}]\mapsto \frac
{z_j\cdot\bar z_k}{z_1\cdot \bar z_1+\cdots + z_{n+1}\cdot \bar z_{n+1}}.$$  
Then the tension field $\tau$ and the conformality operator $\kappa$ on $\cn P^n$ satisfy
$$\tau(\phi_{jk})=-\,4(n+1)\cdot\phi_{jk}\ \ \text{and}\ \ \kappa(\phi_{jk},\phi_{lm})=-\,4\cdot \phi_{jk}\cdot\phi_{lm}.$$
\end{example}

The next result shows that a given eigenfamily $\E$ on $(M,g)$ induces a large collection $\P_d(\E)$ of such objects.  This is particularly interesting in light of Table \ref{table-eigenfamilies}.

\begin{theorem}\cite{Gha-Gud-5}
\label{theorem-polynomials}
Let $(M,g)$ be a Riemannian manifold and the set of complex-valued functions $$\E=\{\phi_i:M\to\cn\,|\,i=1,2,\dots,n\}$$ 
be a finite eigenfamily i.e. there exist complex numbers $\lambda,\mu\in\cn$ such that for all $\phi,\psi\in\E$ $$\tau(\phi)=\lambda\cdot\phi\ \ \text{and}\ \ \kappa(\phi,\psi)=\mu\cdot\phi\,\psi.$$ 
Then the set of complex homogeneous polynomials of degree $d$
$$\P_d(\E)=\{P:M\to\cn\,|\, P\in\cn[\phi_1,\phi_2,\dots,\phi_n],\, P(\alpha\cdot\phi)=\alpha^d\cdot P(\phi),\, \alpha\in\cn\}$$ 
is an eigenfamily on $M$ such that for all $P,Q\in\P_d(\E)$ we have
$$\tau(P)=(d\,\lambda+d(d-1)\,\mu)\cdot P\ \ \text{and}\ \ \kappa(P,Q)=d^2\mu\cdot P\, Q.$$
\end{theorem}

\section{The Cartan Embedding}
\label{section-cartan-embedding}

In this section we describe the Cartan embedding which is our most important tool for the unifying scheme that we are introducing in this work.  For the general theory of Riemannian symmetric spaces we refer to the much celebrated standard text \cite{Hel} of Helgason.
\smallskip

\begin{definition}\label{Cartan map}
Let $(G,K,\sigma)$ be a symmetric triple of compact type. Then the \emph{Cartan map} $\hat \Phi: G \to G$ is defined by
\begin{equation*}
\hat\Phi: p \mapsto p\cdot\sigma(p^{-1}).
\end{equation*}
Since $\hat\Phi(p\cdot k) = \hat\Phi(p)$, for each $k \in K$, $\hat\Phi$ induces a map $\Phi: G/K \to G$ on the symmetric space $G/K$ with
\begin{equation*}
\Phi: pK \mapsto \hat\Phi(p) 
= p\cdot \sigma(p^{-1}),
\end{equation*}
known as the \emph{Cartan embedding}.
\end{definition}

The following result plays an important role in this work. It can be found in \cite{Che-Ebi} as Proposition 3.42.

\begin{theorem}\cite{Che-Ebi}
\label{thm: Cartan embedding properties}
Let $(G,K,\sigma)$ be a symmetric triple of compact type and let $\la{g} = \la{k} \oplus \la{p}$ be the standard decomposition of the Lie algebra $\la{g}$ of $G$. Then the Cartan embedding $\Phi: G/K \to G$ satisfies
\begin{enumerate}
\item[(i)] $\Phi$ is an immersion, 
\item[(ii)] the image of $\Phi$ is totally geodesic in $G$,
\item[(iii)] $\Phi(G/K) = \exp(\la{p})$,
\item[(iv)] $\Phi$ is conformal with constant conformal factor $2$.
\end{enumerate}
\end{theorem}

\begin{remark}
What usually is called the "Cartan embedding" in the literature, is actually in some cases only an immersion and not an embedding.  The best known case is $\Phi:S^n=\SO{n+1}/\SO n\to\SO{n+1}$ mapping the $n$-dimensional unit sphere $S^n$ in $\rn^{n+1}$ onto the totally geodesic image in $\SO{n+1}$. This is diffeomorphic to the real projective space $\rn P^n$.  For this see Example 5.26 of \cite{Gud-Riemann}.
\end{remark}

\begin{proof}(Theorem \ref{thm: Cartan embedding properties})
The proofs of parts (i)-(iii) can be found in \cite{Che-Ebi}.  We have not been able to find the arguments for (iv) in the literature.  For this reason they are provided here.
\smallskip

In order to calculate the differential of the Cartan map at a point $p \in G$ we rewrite $\hat\Phi$ as a composition $\hat\Phi = \mu \circ \xi$, where $\xi: G \to G\times G$ is the map
\begin{equation*}
\xi: p \mapsto (p,\sigma(p^{-1}))
\end{equation*}
and $\mu: G\times G \to G$ is the group multiplication. Let $X_p \in T_p G$ be a tangent vector at $p$ and $X \in \la{g}$ be the corresponding left-invariant vector field with $X_p = (dL_p)_e(X)$. Then we have
\begin{eqnarray*}
d\phi_p(X_p) 
&=&\frac{d}{dt}(\xi(p\cdot \exp(tX)))\big|_{t=0}\\ 
&=&\frac{d}{dt}((p\cdot\exp(tX),\sigma(\exp(-tX)\sigma(p^{-1}))))\big|_{t=0}\\
&=& ((dL_p)_e(X),-(dR_{\sigma(p^{-1})})_e(d\sigma(X))).
\end{eqnarray*}
Employing the chain rule we see that the differential $d\hat\Phi$ of $\hat\Phi$ satisfies
\begin{eqnarray*}
d\hat\Phi_p(X_p)
&=& d\mu_{(p,\sigma(p^{-1}))}((dL_p)_e(X),-(dR_{\sigma(p^{-1})})_e(d\sigma(X)))\\
&=& (dR_{\sigma(p^{-1})})_p(dL_p)_e(X) + (dL_p)_{\sigma(p^{-1})}(dR_{\sigma(p^{-1})})_e(-d\sigma(X))\\
&=& d(R_{\sigma(p^{-1})} \circ L_p)_e (X) + d(L_p \circ R_{\sigma(p^{-1})})_e(-d\sigma(X))\\
&=& d(R_{\sigma(p^{-1})} \circ L_p)_e(X - d\sigma(X)).
\end{eqnarray*}
In the last step we have used the fact that right and left translations commute with each other. For $X \in \la{k}$, we have $d\sigma(X) = X$ and the inner term vanishes as expected. On the other hand for $X \in \la{p}$, we have $d\sigma(X) = -X$ and we yield 
\begin{eqnarray*}
d\hat\Phi_p(X_p)
&=& d(R_{\sigma(p^{-1})} \circ L_p)_e(X - d\sigma(X))\\
&=& d(R_{\sigma(p^{-1})} \circ L_p)_e(2X)\\
&=& 2(dR_{\sigma(p^{-1})})_p(X_p).
\end{eqnarray*}
Since the natural projection $\pi: G \to G/K$ is a Riemannian submersion and the Cartan map $\hat\Phi$ satisfies $\hat\Phi = \Phi \circ\pi$, we only need to show that the Cartan map is conformal when restricted to the horizontal distribution on $G$ induced by the submersion $\pi$. In this case, it is easy to check that the horizontal distribution is given by left-translations of the subspace $\la{p}$ of $\la{g}$. If $p \in G$, $X_p, Y_p$ are horizontal vectors at $p$ and $g$ is the bi-invariant metric on $G$ we then have
\begin{eqnarray*}
g_{\hat\Phi(p)}(d\hat\Phi_p(X_p),d\hat\Phi_p(Y_p)) 
&=& g_{\Phi(p)}(2dR_{\sigma(p^{-1})}(X_p),2dR_{\sigma(p^{-1})}(Y_p))\\
&=& 4 \cdot g_p(X_p,Y_p),
\end{eqnarray*}
where the last equality follows from the right-invariance of the metric.
\end{proof}

\begin{remark}
The Cartan embedding $\Phi:G/K\to G$ is a conformal immersion with constant conformal factor and the image $\Phi(G/K)$ is totally geodesic in $G$. It then follows from the Proposition on page 119, of the seminal work \cite{Eel-Sam} of Eells and Sampson, that $\Phi$ is a harmonic map.
\end{remark}

\section{Eigenfamilies via the Cartan Embedding}
\label{sec: Eigenfamilies from Cartan Embedding}

In this section we show how the Cartan embedding can be used to construct eigenfunctions and eigenfamilies. For this we employ the following well-known result.

\begin{theorem}
\label{thm: Cartan Harmonic}
Let $(G,K,\sigma)$ be a symmetric triple of compact type with and $\dim(G/K) \geq 2$. Then the Cartan map $\hat\Phi: G \mapsto G$ is harmonic i.e. it satisfies
\begin{equation*}
\tau(\hat\Phi) = 0.
\end{equation*}
\end{theorem}
\begin{proof}
The projection $\pi :G \to G/K$ is a Riemannian submersion with totally geodesic fibres, and hence a harmonic morphism (see \cite{Fug,Ish} and \cite{Bai-Eel}). It then follows from Proposition 7.13 in \cite{Eel-Lem} that the composition $\hat{\Phi} = \Phi \circ \pi$ is harmonic as well.
\end{proof}

When combined with the composition law, this has the following consequences given by equations (\ref{eq: tau f o Phi}) and (\ref{eq: kappa f o Phi}).  We let $\phi:G \to \cn$ be a smooth complex-valued function and also denote by $\phi$ the restriction of $\phi$ to the image $\hat\Phi(G)$. Applying the composition law to the function $\phi\circ \hat\Phi: G \to \cn$ then gives
\begin{eqnarray}
\label{eq: tension of composition}
\tau(\phi\circ \hat\Phi) 
&=& d\phi(\tau(\hat\Phi))\circ \hat\Phi + \trace \nabla d\phi(d\hat\Phi,d\hat\Phi) \circ \hat\Phi \notag\\
&=& \trace \nabla d\phi(d\hat\Phi,d\hat\Phi) \circ \hat\Phi,
\end{eqnarray}
since $\hat\Phi$ is harmonic. We now choose an orthonormal basis
$$\basis_\la{g} = \{X_1,\dots, X_k, Y_1, \dots, Y_l\}$$
for $\la{g}$ with $\basis_\la{k}  = \{X_1,\dots, X_k\}$ and $\B_\p=\{Y_1,\dots, Y_l\}$ orthonormal bases for $\la{k}$ and $\la{p}$, respectively. Expanding the right hand side of equation \eqref{eq: tension of composition}, we obtain
\begin{equation*}
\trace \nabla d\phi(d\hat\Phi,d\hat\Phi) \circ \hat\Phi 
= \Big\{\sum_{X \in \mathcal{B}_{\la{g}}} d\hat\Phi(X)^2(\phi) - \big(\nab{d\hat\Phi(X)}{d\hat\Phi(X)}\big)(\phi)\Big\}\circ \hat\Phi.
\end{equation*}

It can be seen in Section \ref{section-GLC}, that the second term of the sum vanishes, and for the first term, only the vector fields from $\mathcal{B}_{\la{p}}$ contribute. This implies that 
\begin{equation}\label{eq: tau f o Phi}
\tau(\phi\circ \hat\Phi)  
= \Big\{\sum_{X \in \mathcal{B_{\la{p}}}} d\hat\Phi(X)^2(\phi)\Big\} \circ \hat\Phi.
\end{equation}
We get a similar formula for the conformality operator. Let $\phi, \psi: G \mapsto \cn$ be smooth. Then
\begin{equation}\label{eq: kappa f o Phi}
\kappa(\phi\circ \hat\Phi, \psi \circ \hat\Phi) = \Big\{\sum_{X\in\mathcal{B}_{\la{p}}}d\hat\Phi(X)(\phi)\cdot d\hat\Phi(X)(\psi)\Big\}\circ \hat\Phi.
\end{equation}

\begin{remark}
\label{rem: basis cartan embedding}
At a point $p \in G$ and for $X \in \la{p}$ we have the expression
\begin{equation*}
d\hat\Phi_p(X_p) = 2(dR_{\sigma(p^{-1})})_p(X_p).
\end{equation*}
Since we assume $G$ to be compact and equipped with a bi-invariant metric then the right translations are isometries. Hence if $\basis_{p}$ is an orthonormal basis for $\la{p}_{p}$, the set
$$
\{ (dR_{\sigma(p^{-1})})_p(X_p) \ | \ X_p \in \basis_{p} \}
$$
is an orthonormal basis for $T_{\hat\Phi(p)}N$, where $N$ is the image $\hat\Phi(G)$ of $\hat\Phi$.
\end{remark}

The next result then follows by Remark \ref{rem: basis cartan embedding}

\begin{proposition}\label{pro: composition relations}
Let $(G,K,\sigma)$ be a symmetric triple of compact type and let $\hat\Phi: G \to G$ be the associated Cartan map. Let $N$ be the image $\hat\Phi(G)$ of $\hat\Phi$, $\tau$, $\kappa$ denote the tension field and conformality operator on the Lie group $G$ and $\tau_{N}$, $\kappa_{N}$ denote the same on $N$. If 
$\phi, \psi: G \to \cn$
are complex-valued $C^2$-functions on $G$, then the compositions 
$$
\phi \circ \hat\Phi,\psi \circ \hat\Phi: G \to \cn
$$
are $K$-invariant $C^2$-functions on $G$. The tension fields and conformality operators then satisfy the relations
\begin{equation*}
\tau(\phi\circ \hat\Phi) = 4\cdot \tau_{N}(\phi)\circ \hat\Phi
\end{equation*}
and
\begin{equation*}
\kappa(\phi\circ \hat\Phi,\psi \circ \hat\Phi) 
= 4\cdot \kappa_{N}(\phi,\psi)\circ \hat\Phi.
\end{equation*}
\end{proposition}

The following is useful when explicitly calculating the Laplace-Beltrami and conformality operators on the image of the Cartan embedding.

\begin{remark}\label{rem: expression for basis}
Let $(G,K,\sigma)$ be a symmetric triple of compact type and $\hat\Phi:G \to G$ be the associated Cartan map. Let $N$ be the image $\hat\Phi(G)$ of $\hat\Phi$. By Remark \ref{rem: basis cartan embedding}, given a point $p\in G$ and an orthonormal basis $\basis_p$ for $\la{p}_p$, an orthonormal basis for $T_{\Phi(p)} N$ is given by
$$
\{ (dR_{\sigma(p^{-1})})_p(X_p) \ | \ X_p \in \basis_{p} \}.
$$
If we now fix $X_p \in \basis_p$ and take $X \in \la{p}$ such that $(dL_p)_e(X) = X_p$ we have
\begin{eqnarray*}
(dR_{\sigma(p^{-1})})_p(X_p) &=& (dR_{\sigma(p^{-1})})_p(dL_p)_e(dL_{\sigma(p^{-1})})_{\sigma(p)}(dL_{\sigma(p)})_{e}(X)\\
&=& (dL_{\hat\Phi(p)})_e \Ad_{\sigma(p)}(X).
\end{eqnarray*}
Hence given an orthonormal basis $\basis_{\la{p}}$ for $\la{p}$ we can construct an orthonormal basis for $T_{\hat\Phi(p)}N$ as
\begin{eqnarray*}
    \basis_{\hat\Phi(p)} = \{\Ad_{\sigma(p)}(X)_{\hat\Phi(p)} \ | \ X \in \basis_{\la{p}} \}.
\end{eqnarray*}
\end{remark}

\section{The General Linear Group $\GLC n$}
\label{section-GLC}

In this section we now turn our attention to the concrete Riemannian matrix Lie groups embedded as subgroups of the complex general linear group.
\medskip

The group of linear automorphism of $\cn^n$ is the complex general linear group $\GLC n=\{ z\in\cn^{n\times n}\,|\, \det z\neq 0\}$ of invertible $n\times n$ matrices with its standard representation 
$$z\mapsto
\begin{bmatrix}
	z_{11} & \cdots & z_{1n} \\
	\vdots & \ddots & \vdots \\
	z_{n1} & \cdots & z_{nn}
\end{bmatrix}.
$$
Its Lie algebra $\glc n$ of left-invariant vector fields on $\GLC n$ can be identified with $\cn^{n\times n}$ i.e. the complex linear space of $n\times n$ matrices.  We equip $\GLC n$ with its natural left-invariant Riemannian metric $g$ induced by the standard Euclidean inner product $\glc n\times\glc n\to\rn$ on its Lie algebra $\glc n$ satisfying
$$g(Z,W)\mapsto \Re\,\trace\, (Z\cdot \bar W^t).$$ 
For $1\le i,j\le n$, we shall by $E_{ij}$ denote the element of $\rn^{n\times n}$ satisfying
$$(E_{ij})_{kl}=\delta_{ik}\delta_{jl}$$ and by $D_t$ the diagonal
matrices $D_t=E_{tt}.$ For $1\le r<s\le n$, let $X_{rs}$ and
$Y_{rs}$ be the matrices satisfying
$$X_{rs}=\frac 1{\sqrt 2}(E_{rs}+E_{sr}),\ \ Y_{rs}=\frac
1{\sqrt 2}(E_{rs}-E_{sr}).$$
For the real vector space $\glc n$ we then have the canonical orthonormal basis $\B^\cn=\B\cup i\B$, where 
$$\B=\{Y_{rs}, X_{rs}\,|\, 1\le r<s\le n\}\cup\{D_{t}\,|\, t=1,2,\dots,n\}.$$
The following relations shall prove useful.
\begin{lemma}\label{lem: square sum relation}
Let $n > 1$ be an integer. Then we have
\begin{align*}
    \sum_{\mathclap{1\leq r < s \leq n}} Y_{rs}^2 = -\tfrac{n-1}{2}\cdot I_{n}, \quad \sum_{\mathclap{1\leq r < s \leq n}} X_{rs}^2 = \tfrac{n-1}{2}\cdot I_{n}, \quad \sum_{t=1}^{n}D_t^2 = I_{n}.
\end{align*}
\end{lemma}
\vskip .1cm

Let $G$ be a classical Lie subgroup of $\GLC n$ with Lie algebra $\g$ inheriting the induced left-invariant Riemannian metric, which we shall also denote by $g$.  In the cases considered in this paper, $\B_{\g}=\B^\cn\cap\g$ will be an orthonormal basis for the subalgebra $\g$ of $\glc n$.  By employing the Koszul formula for the Levi-Civita connection $\nabla$ on $(G,g)$, we see that for all $Z,W\in\B_{\g}$ we have
\begin{eqnarray*}
	g(\nab ZZ,W)&=&g([W,Z],Z)\\
	&=&\Re\,\trace\, ((WZ-ZW)\bar Z^t)\\
	&=&\Re\,\trace\, (W(Z\bar Z^t-\bar Z^tZ))\\
	&=&0.
\end{eqnarray*}

If $Z\in\g$ is a left-invariant vector field on $G$ and $\phi:U\to\cn$ is a local complex-valued function on $G$ then the $k$-th order derivatives $Z^k(\phi)$ satisfy
\begin{equation*}\label{equation-diff-Z}
	Z^k(\phi)(p)=\frac {d^k}{ds^k}\bigl(\phi(p\cdot\exp(sZ))\bigr)\Big|_{s=0},
\end{equation*}
This implies that the tension field $\tau$ and the conformality operator $\kappa$ on $G$ fulfill 
\begin{equation*}\label{equation-tau}
	\tau(\phi)
	=\sum_{Z\in\B_\g}\bigl(Z^2(\phi)-\nab ZZ(\phi)\bigr)
	=\sum_{Z\in\B_\g}Z^2(\phi),
\end{equation*}	
\begin{equation*}\label{equation-kappa}
	\kappa(\phi,\psi)=\sum_{Z\in\B_\g}Z(\phi)\cdot Z(\psi),
\end{equation*}
where $\B_\g$ is the orthonormal basis $\B^\cn\cap\g$ for the Lie algebra $\g$.

\section{The Quaternionic Unitary Group}
\label{section-quaternionic-unitary-group}

For a positive integer $n\in\zn^+$, the quaternionic unitary group $\Sp n$ is the intersection of the unitary group $\U{2n}$ and the standard representation of the quaternionic general linear group $\GLH n$ in $\cn^{2n\times 2n}$ given by
\begin{equation}\label{equation-Spn}
(z+jw)\mapsto q=
\begin{bmatrix}
z & w \\
 -\bar w & \bar z
\end{bmatrix}.
\end{equation}
The Lie algebra $\sp n$ of $\Sp n$ satisfies

$$\sp{n}=\left\{\begin{bmatrix} Z & W
\\ -\bar W & \bar Z\end{bmatrix}\in\cn^{2n\times 2n}
\ \Big|\ Z^*+Z=0,\ W^t-W=0\right\}.$$ 
The following elements form the orthonormal basis $\B_{\sp n}$ for the Lie algebra $\sp n$ of $\Sp n$
$$Y^a_{rs}=\frac 1{\sqrt 2}
\begin{bmatrix}
Y_{rs} & 0 \\
     0 & Y_{rs}
\end{bmatrix}, 
X^a_{rs}=\frac 1{\sqrt 2}
\begin{bmatrix}
iX_{rs} & 0 \\
      0 & -iX_{rs}
\end{bmatrix},$$
$$ X^b_{rs}=\frac 1{\sqrt 2}
\begin{bmatrix}
      0 & iX_{rs} \\
iX_{rs} & 0\end{bmatrix}, 
X^c_{rs}=\frac 1{\sqrt 2}
\begin{bmatrix}
      0 & X_{rs} \\
-X_{rs} & 0
\end{bmatrix},$$
$$D^a_{t}=\frac 1{\sqrt 2}
\begin{bmatrix}
iD_{t} & 0 \\
     0 & -iD_{t}
\end{bmatrix}, 
D^b_{t}=\frac 1{\sqrt 2}
\begin{bmatrix}
     0 & iD_{t}  \\
iD_{t} & 0
\end{bmatrix}, 
D^c_{t}=\frac 1{\sqrt 2}
\begin{bmatrix}
     0 & D_{t}  \\
-D_{t} & 0
\end{bmatrix}.$$
Here $1\le r<s\le n$ and $1\le t\le n$.
\smallskip 

For those interested in the construction of eigenfamilies on the quaternionic unitary group we recommend the paper \cite{Gud-Sak-1}.

\section{The Quaternionic Grassmannians}
\label{section-quaternionic-Grassmannians}

In this section we describe the compact quaternionic Grassmannians which are classical Riemannian symmetric spaces.
\smallskip

For positive integers $m,n\in\zn^+$, let $G$ be the quaternionic unitary group $\Sp{m+n}$ with the subgroup 
$K=\Sp{m}\times \Sp{n}$ given in terms of quaternionic matrices by 
\begin{equation*}
\Sp{n}\times\Sp{m}= 
\left\{ 
\begin{bmatrix}
q_1 & 0\\
0 & q_2
\end{bmatrix}
\in \Sp{m+n}\ 
\Big | \ q_1 \in \Sp{m}, \ q_2 \in \Sp{n}
\right\}.
\end{equation*}
Let $I_m\in\rn^{m\times m}$ denote the identity matrix and define $I_{m,n}\in\rn^{(m+n)\times (m+n)}$ and 
$\hat{I}_{m,n}\in\rn^{2(m+n)\times 2(m+n)}$ by 
$$
I_{m,n}=
\begin{bmatrix}
I_m & 0\\
0 & -I_n
\end{bmatrix}\ \ \text{and}\ \ 
\hat{I}_{m,n} 
=\begin{bmatrix}
I_{m,n} & 0\\
0 & I_{m,n}
\end{bmatrix},
$$
respectively.  Let $\sigma:G\to G$ be the involutive automorphism with $\sigma(q) = I_{m,n}\cdot q\cdot I_{m,n}$. Then it is easily seen that the subgroup $K$ is the fixed-point set of $\sigma$. Thus the compact quotient  
$$
\Sp{m+n}/\Sp{m}\times \Sp{n}
$$
is a Riemannian symmetric space, called the {\it quaternionic Grassmannian}. In terms of the standard complex representation of $\Sp{m+n}$ one easily checks that $\sigma(q) = \hat{I}_{m,n} \cdot q\cdot \hat{I}_{m,n}$. The Cartan map $\hat\Phi:G\to G$ is given by
\begin{equation*}
\hat\Phi: q \mapsto q\cdot \hat{I}_{m,n}\cdot \bar{q}^t\cdot \hat{I}_{m,n}
\end{equation*}
and is clearly $K$-invariant.  This induces the Cartan embedding $\Phi:G/K\to G$ with $\Phi(qK)=\hat\Phi(q)$.

Let $\sp{m+n}  = \k \oplus \p$ be the Cartan decomposition induced by $\sigma$. Then $\k \cong \sp{m} \oplus \sp{n}$ and the set
\begin{equation*}
    \left\{Y^{a}_{rs}, X^{a}_{rs}, X^{b}_{rs}, X^{c}_{rs} \ \middle| \ m+1 \leq r \leq m+n, 1 \leq s \leq m \right\}
\end{equation*}
is an orthonormal basis for $\p$.
\section{Eigenfamilies on The Quaternionic Grassmannians}
\label{sec: Eigenfamilies on Quaternionic Grassmannians}

The aim of this section is to construct eigenfamilies on the quaternionic Grassmannians.  Here we employ the Cartan embedding to obtain known examples but also new ones, see Theorem \ref{thm: eigenfamilies on QG's}.
\smallskip

For $1 \leq j,\alpha \leq 2(m+n)$, we define the maps $\hat\psi_{j\alpha}: \Sp{m+n} \to \cn$ on the quaternionic group $\Sp{m+n}$ by
\begin{equation*}
 \hat\psi_{j\alpha}: q \mapsto \tfrac{1}{2}\cdot(q\hat{I}_{m,n}\bar{q}^t + I_{2(m+n)})_{j\alpha}.
\end{equation*}
 It is easy to verify that $\hat\psi_{j\alpha}$ then satisfies
\begin{equation}\label{eq: psi as poly}
\hat\psi_{j\alpha}(q) = \sum_{r=1}^{m} q_{jr}\bar{q}_{\alpha r} + \sum_{r= m+n+1}^{2m + n} q_{jr}\bar{q}_{\alpha r}.
\end{equation}
Further we see that we may write $\hat\psi_{j\alpha}$ as a composition with the Cartan map $\hat\psi_{j\alpha} = \eta_{j\alpha} \circ \hat\Phi$, where $\eta_{j\alpha}: \Sp{m+n} \to \cn$ is given by
\begin{equation*}
    \eta_{j\alpha}: q \mapsto \tfrac{1}{2}\cdot(q\cdot \hat{I}_{m,n} + I_{2(m+n)})_{j\alpha}.
\end{equation*}
Hence $\hat\psi_{j\alpha}$ is $\Sp{m}\times \Sp{n}$-invariant.

\begin{proposition}\label{prop: Eigenfamilies on QG}
Let $1 \leq j,k,\alpha  \leq 2(m+n)$ with $j,k \neq \alpha$ . Then the tension field $\tau$ and the conformality operator $\kappa$ on the quaternionic unitary group $\Sp{m+n}$ satisfy
\begin{equation*}
\tau(\hat\psi_{j\alpha}) = -2(m+n)\cdot \hat\psi_{j\alpha} \ \ \text{and} \ \  \kappa(\hat\psi_{j\alpha},\hat\psi_{k\alpha}) = -\hat\psi_{j\alpha}\hat\psi_{k\alpha}.
\end{equation*}
\end{proposition}

\begin{proof}
Denote by $N$ the totally geodesic image of the Cartan embedding $\Phi$ in $\Sp{m+n}$. It follows from Proposition \ref{pro: composition relations} that in order to compute $\tau(\hat\psi_{j\alpha})$ and $\kappa(\hat\psi_{j\alpha},\hat\psi_{k\beta})$ at a point $q \in \Sp{m+n}$ it is sufficient to compute 
$$\tau_N (\eta_{j\alpha})(\hat\Phi(q))\ \ \text{and}\ \  \kappa_N(\eta_{j\alpha},\eta_{k\beta})(\hat\Phi(q))$$ 
at $\hat\Phi(q)\in N$. By Remark \ref{rem: expression for basis}, for each $q\in\Sp{m+n}$ an orthonormal basis for $T_{\hat\Phi(q)} N$ is given by
\begin{equation*}
\basis_{\hat\Phi(q)} = \{\Ad_{\sigma(q)}(Q)_{\hat\Phi(q)} \ | \ Q \in \basis_\la{p} \}.
\end{equation*}
We then have
\begin{align*}
&\Ad_{\sigma(q)}(Q)(\eta_{j\alpha})(\hat\Phi(q)) \\
&= \dtatzero \eta_{j\alpha}(\hat\Phi(q)\cdot \exp(t\Ad_{\sigma(q)}Q))\\
&= \frac{1}{2}\dtatzero e_j \cdot (\hat\Phi(q)\sigma(q)\exp(tQ)\sigma(q)^{-1}\hat{I}_{m,n} + I_{2(m+n)})\cdot e_\alpha^t\\
&= \frac{1}{2}\dtatzero e_j \cdot (q\exp(tQ)\hat{I}_{m,n}\bar{q}^t + I_{2(m+n)})\cdot e_\alpha^t\\ 
&= \frac{1}{2}e_j\cdot q\,Q\hat{I}_{m,n}\bar{q}^t \cdot e_\alpha^t.
\end{align*}
Similarily we compute
\begin{eqnarray*}
\Ad_{\sigma(q)}(Q)^2(\eta_{j\alpha})(\hat\Phi(q)) 
&=& \restr{\frac{d^2}{dt^2}}{t = 0} \eta_{j\alpha}(\hat\Phi(q)\cdot \exp(t\Ad_{\sigma(q)}Q))\\
&=& \frac{1}{2}e_j\cdot q\,Q^2\hat{I}_{m,n}\,\bar{q}^t \cdot e_\alpha^t 
\end{eqnarray*}

For the tension field $\tau_N$ we then have
\begin{align}
\begin{split}
&\quad\tau_N(\eta_{j\alpha})(\Phi(q))\\ 
&= \frac{1}{2}\sum_{Q \in \basis_{\la{p}}} e_j\cdot qQ^2\tilde{I}_{m,n}\bar{q}^t \cdot e_\alpha^t\\
&= \frac{1}{2}\sum_{r=1}^m \sum_{s=m+1}^{m+n}e_j\cdot q \left( (Y^{a}_{rs})^2 + (X_{rs}^a)^2 + (X_{rs}^b)^2 + (X_{rs}^c)^2\right)
\tilde{I}_{m,n}\bar{q}^t \cdot e_\alpha^t\\
&= \frac{1}{4} e_j\cdot q
\left[ \sum_{r=1}^m \sum_{s=m+1}^{m+n}
\begin{psmallmatrix}
Y_{rs}^2 & 0\\
0 & Y_{rs}^2
\end{psmallmatrix}
- 3
\begin{psmallmatrix}
X_{rs}^2 & 0\\
0 & X_{rs}^2
\end{psmallmatrix}\right]
\tilde{I}_{m,n}\bar{q}^t \cdot e_\alpha^t\\
&= e_j\cdot q
\left[ \sum_{r=1}^m \sum_{s=m+1}^{m+n}
\begin{psmallmatrix}
Y_{rs}^2 & 0\\
0 & Y_{rs}^2
\end{psmallmatrix}\right]
\tilde{I}_{m,n}\bar{q}^t \cdot e_\alpha^t.   
\end{split}
\end{align}
From Lemma \ref{lem: square sum relation}, we can deduce
    \begin{align*}
        \sum_{r=1}^m \sum_{s=m+1}^{m+n}\begin{psmallmatrix}
        Y_{rs}^2 & 0\\
        0 & Y_{rs}^2
        \end{psmallmatrix} =-\frac{m+n}{4}\cdot I_{2(m+n)} + \frac{m-n}{4}\cdot \tilde{I}_{m,n},
    \end{align*}
     which yields
\begin{align*}
\begin{split}
&\quad\tau_N(\eta_{j\alpha})(\Phi(q))\\
&= e_j\cdot q\left[-\frac{m+n}{4}\cdot I_{2(m+n)} + \frac{m-n}{4}\cdot \tilde{I}_{m,n}\right]\tilde{I}_{m,n}\bar{q}^t\cdot e^t_{\alpha}\\
&= -\frac{m+n}{2}\cdot \frac{1}{2}e_j\cdot(q\tilde{I}_{m,n}\bar{q}^t + I_{2(m+n)})\cdot e_{\alpha}^t + \frac{m+n}{4}\cdot\delta_{j\alpha} + \frac{m-n}{4}\cdot\delta_{j\alpha}\\
&= -\frac{m+n}{2}\cdot \psi_{j\alpha}(q) + 2m\cdot\delta_{j\alpha}.
\end{split}
\end{align*}
Under our assumption that $j \neq \alpha$ this simplifies further to 
\begin{equation}\label{eq: LB on N 2}
\tau_N(\eta_{j\alpha})(\hat\Phi(q)) = -\frac{m+n}{2}\cdot \hat\psi_{j\alpha}(q).
\end{equation}

We now turn to the conformality operator. For $1 \leq j,k,\alpha,\beta \leq$ we have
\begin{align*}
&\quad\kappa_{N}(\eta_{j\alpha},\eta_{k\beta})(\Phi(q))\\
&= \frac{1}{4}\sum_{r=1}^{m}\sum_{s= m+1}^{m+n}\bigg[e_j\cdot q Y_{rs}^a
\tilde{I}_{m,n}\bar{q}^t\cdot e_{\alpha}^t\cdot e_k\cdot q
Y_{rs}^a\tilde{I}_{m,n} \bar{q}^t\cdot e_{\beta}^t\\[-0.3cm]
&\sed\sed\sed\quad+ e_j\cdot q X_{rs}^a
\tilde{I}_{m,n}\bar{q}^t\cdot e_{\alpha}^t\cdot e_k\cdot q
X_{rs}^a\tilde{I}_{m,n} \bar{q}^t\cdot e_{\beta}^t\\
&\sed\sed\sed\quad  + e_j\cdot q x_{rs}^b
\tilde{I}_{m,n}\bar{q}^t\cdot e_{\alpha}^t\cdot e_k\cdot q
X_{rs}^b\tilde{I}_{m,n} \bar{q}^t\cdot e_{\beta}^t\\
&\sed\sed\sed\quad  + e_j\cdot q 
X_{rs}^c
\tilde{I}_{m,n}\bar{q}^t\cdot e_{\alpha}^t\cdot e_k\cdot q
X_{rs}^c\tilde{I}_{m,n} \bar{q}^t\cdot e_{\beta}^t\bigg].
\end{align*}
Expanding this as a polynomial in the matrix components of $q$ and $\bar{q}$ and making use of the notation $\tilde{r} = r + m+n$ and $\tilde{s} = s+ m +n$ we arrive, after a lengthy calculation, at
\begin{align*}
&\kappa_{N}(\eta_{j\alpha},\eta_{k\beta})(\Phi(q))\\
 &=\frac{1}{8}\cdot\left(\delta_{j\beta}\cdot\hat\psi_{k\alpha}(q) +\delta_{j\alpha}\cdot\hat\psi_{j\beta}(q) - 2\cdot\hat\psi_{j\beta}(q)\cdot\hat\psi_{k\alpha}(q)\right)\\
&\quad+ \frac{1}{8}\,\sum_{r=1}^{m}\sum_{s = m+1}^{m+n}\Big[ (q_{jr}q_{k\tilde{r}}-q_{j\tilde{r}}q_{kr})(\bar{q}_{\alpha s}\bar{q}_{\beta \tilde{s}} - \bar{q}_{\alpha \tilde{s}}\bar{q}_{\beta s})\\[-0.4cm]
&\qquad\qquad\qquad\qquad  +(q_{j\tilde{s}}q_{ks}-q_{js}q_{k\tilde{s}})(\bar{q}_{\alpha \tilde{r}}\bar{q}_{\beta r} - \bar{q}_{\alpha r}\bar{q}_{\beta \tilde{r}})\Big].
\end{align*}
Recalling our assumption that $\alpha = \beta \neq j, k$, we see that the terms involving sums and Kronecker deltas cancel and we conclude
\begin{equation}\label{eq: kappa}
\kappa_{N}(\eta_{j\alpha},\eta_{k\alpha})(\hat\Phi(q)) = -\frac{1}{4}\cdot\hat\psi_{j\alpha}(q)\cdot\hat\psi_{k\alpha}(q).
\end{equation}
The result now follows by applying Proposition  \ref{pro: composition relations} to the equations \eqref{eq: LB on N 2} and \eqref{eq: kappa}.
\end{proof}

The $\Sp{m}\times \Sp{n}$-invariant functions $\hat \psi_{j\alpha}: \Sp{m+n} \to \cn$ induce functions ${\psi}_{j\alpha}: \Sp{m+n}/\Sp{m}\times \Sp{n} \to \cn$ on the quotient space. The eigenfunctions obtained by taking $1 \leq j,\alpha \leq m+n$ coincide with those found by Gudmundsson and Ghandour in \cite{Gha-Gud-5}, while the remaining cases are new. Hence, by the following result, we yield the eigenfamilies found in \cite{Gha-Gud-5} and obtain new ones.

\begin{theorem}\label{thm: eigenfamilies on QG's}
For a fixed natural number $1 \leq \alpha \leq 2(m+n)$ the set
\begin{equation*}
\mathcal{E}_{\alpha} = \{{\psi}_{j\alpha}:\Sp{m+n}/\Sp{n}\times \Sp{m} \to \cn \ | \ 1 \leq j \leq 2(m+n),\ j \neq \alpha\}
\end{equation*}
is an eigenfamily on the quaternionic Grassmannian such that the tension field $\tau$ and conformality operator $\kappa$ satsify
\begin{equation*}
\tau({\phi}) = -2(m+n)\cdot {\phi} \ \ \text{and} \ \  \kappa({\phi},{\psi}) = -\,{\phi}\cdot{\psi},
\end{equation*}
for all ${\phi},{\psi} \in \mathcal{E}_{\alpha}$.
\end{theorem}

\section{The Complex Grassmannians}

As a homogenous space, the complex Grassmannians are obtained as $G_{n}(\cn^{m+n}) = \U{m+n}/\U{m}\times \U{n}$. The Cartan involution of $\U{m + n}$ is given by conjugation with the block matrix $I_{m,n}$. So explicitly $$\sigma(z) = I_{m,n}\cdot z \cdot I_{m,n}.$$ The Cartan map is thus $\Phi(z) = z\cdot I_{m,n}\cdot \bar{z}^t \cdot I_{m,n}$.
For $1\leq j,\alpha \leq m+n$ set
$$
\eta_{j\alpha}(z) = (z\cdot I_{m,n} + I_{m+n})_{j\alpha}
$$
and let $\hat\psi_{j\alpha} = \eta_{j\alpha} \circ \Phi$. Then one easily checks that these functions coincide with the $\U{m}\times \U{n}$-invariant functions described in \cite{Gha-Gud-5}. Denoting by ${\psi}_{j\alpha}$ the induced functions on the Grassmannians, we thus have the following.

\begin{theorem}\cite{Gha-Gud-5}
\label{eigenfams comp grass}
Let $M$ be the complex Grassmannian manifold  $$\U{m+n}/\U{m}\times \U{n}.$$
For a fixed natural number $1\leq \alpha\leq m+n$, the set
\begin{equation*}
\mathcal{E}_{\alpha} = \{ {\psi}_{j\alpha}: M \to \cn \ | \ 1 \leq j \leq m+n, j \neq \alpha \} 
\end{equation*}
is an eigenfamily on the complex Grassmannian $M$ such that the tension field $\tau$ and conformality operator $\kappa$ satisfy
\begin{equation*}
\tau(\phi) = -2(m+n)\cdot \phi
\ \ \text{and}\ \  
\kappa(\phi,\psi) = -2\cdot \phi\,\psi,
\end{equation*}
for all $\phi,\psi \in \mathcal{E}_{\alpha}$.
\end{theorem}

\section{The Remaining Classical Compact Symmetric Spaces}

A detailed description of the Riemannian symmetric spaces of this section can be found in Helgason's book \cite{Hel}. There the Cartan involutions are also listed, from which the Cartan embeddings are easily derived. 
\smallskip

The eigenfunctions on the real Grassmannians found in \cite{Gha-Gud-4} fall into our unifying scheme as follows.

\begin{theorem}
\label{thm: real grass}
Let $M$ be the real Grassmannian of oriented $n$-planes in $\rn^{m+n}$ i.e. 
$$\SO{m+n}/\SO{m}\times \SO{n}.$$ 
Let $A$ be an $(m+n)\times (m+n)$ complex symmetric matrix with $\mathrm{rank} A = 1$, $A^2 = 0$ and $\trace A = 0$. Define $\eta_{A}:\SO{m+n}\to\cn$ by
$$
\eta_{A}:x\mapsto \tfrac{1}{2}\trace((x\cdot I_{m,n} + I_{m+n})\cdot A).
$$
Then the function $\psi_A: M \to \cn$ induced by the composition $\hat\psi_A = \eta_A \circ \hat\Phi$ of $\eta_A$ with the Cartan map is an eigenfuntion on $M$ satisfying
\begin{equation*}
\tau(\psi_A) 
= -(m+n)\cdot\psi_A
\ \ \text{and}\ \  \kappa(\psi_A,\psi_A) 
= -2\cdot\psi_A^2.
\end{equation*}
\end{theorem}

The same applies for the solutions on the Riemannian symmetric space $\SU{n}/\SO{n}$ constructed in \cite{Gud-Sif-Sob-2} with the following. 

\begin{theorem}
\label{thm: SU/SO}
Let $M$ be the Riemannian symmetric space $\SU{n}/\SO{n}$. Let $A = a^ta \in \cn^{n\times n}$ for some non-zero complex vector $a \in \cn^n$ and define $\eta_A:\SU{n}\to\cn$ with 
$$
\eta_A:z\mapsto\trace(z\cdot A).
$$
Then the function $\phi_A: M \to \cn$ induced by the composition $\hat\phi_A = \eta_A \circ \Phi$ of $\eta_A$ with the Cartan map is an eigenfunction on $M$ satisfying
\begin{equation*}
\tau(\phi_A) 
= -\frac{2(n^2 +n -2)}{n}\cdot\phi_A, 
\quad 
\kappa(\phi_A,\phi_A) 
= -\frac{4(n-1)}{n} \cdot\phi_A^2.
\end{equation*}
\end{theorem}

In their work \cite{Gud-Sif-Sob-2} the authors provide complex-valued eigenfunctions on $\SO{2n}/\U{n}$.  In our unifying scheme the are described as follows.

\begin{theorem}
\label{thm: SO/U}
Let $M$ be the Riemannian symmetric space $\SO{2n}/\U{n}$. Let $V$ be an isotropic subspace of $\cn^{2n}$, $a,b\in V$ be linearly independent and $A$ the skew-symmetric matrix
$$
A = \sum_{r,s = 1}^{2n} a_r b_s Y_{rs}.
$$
Define the complex-valued function $\eta_A:\SO{2n} \mapsto\cn$ by
$$
\eta_A:x\mapsto \trace(x\cdot JA).
$$
Then the function $\psi_A: M \to \cn$ induced by the composition $\hat\psi_A = \eta_A \circ \Phi$ of $\eta_A$ with the Cartan map is an eigenfunction on $M$ satisfying
\begin{equation*}
\tau(\psi_A) 
= -2(n-1)\cdot\psi_A, 
\quad 
\kappa(\psi_A,\psi_A) 
= - \psi_A^2.
\end{equation*}
\end{theorem}

The same applies to the eigenfunctions on $\Sp{n}/\U{n}$ found again in \cite{Gud-Sif-Sob-2}.

\begin{theorem}
\label{thm: Sp/U}
Let $M$ be the Riemannian symmetric space $\Sp{n}/\U{n}$. Further let $A = a^t a$ for some non-zero vector $a \in \cn^{2n}$ and define the complex-valued function $\eta_A:\Sp{n}\to\cn$ with
$$
\eta_A:q\mapsto\trace(q\cdot A).
$$
Then the function $\phi_A: M \to \cn$ induced by the composition $\hat\phi_A = \eta_A\circ \Phi$ of $\eta_A$ with the Cartan map is an eigenfunction on $M$ satisfying
\begin{equation*}
\tau(\phi_A) 
= -2(n+1)\cdot\phi_A
\ \ \text{and}\ \ 
\kappa(\phi_A,\phi_A) = -2\cdot\phi_A^2.
\end{equation*}
\end{theorem}

In the last case of $\SU{2n}/\Sp{n}$ the eigenfunctions produced in \cite{Gud-Sif-Sob-2} can be described as follows.

\begin{theorem}
\label{thm: SU/Sp}
Let $M$ be the Riemannian symmetric space $\SU{2n}/\Sp{n}$. For two non-zero linearly independent vectors $a,b \in \cn^{2n}$, let $A\in \cn^{2n\times 2n}$ be the skew-symmetric matrix
$$
A = \sum_{r,s = 1}^{2n}a_r b_s Y_{rs}.
$$
Define the complex-valued function $\eta_A:\SU{2n} \to\cn$ by
$$
\eta_A: z\mapsto \trace(z\cdot A).
$$
Then the function $\phi_A: M \to \cn$ induced by the composition $\hat\phi_A = \eta_A\circ \Phi$ of $\eta_A$ with the Cartan map is an eigenfunction on $M$ satisfying
\begin{equation*}
\tau(\phi_A) 
= \frac{2(2n^2-n-1)}{n}\cdot\phi_A 
\ \ \text{and}\ \ 
\kappa(\phi_A,\phi_A) 
= -\frac{2(n-1)}{n}\cdot\phi_A^2.
\end{equation*}
\end{theorem}

\section{A New Constructions}
\label{section-new-constructions}

In this section we introduce a new method for the construction complex-valued harmonic morphisms on the product of Riemannian manifolds.  The main result is given by the following statement. 
 
\begin{theorem}\label{thm: eigenfunctions on products}
Let $(M_1,g_1)$ and $(M_2,g_2)$ be two Riemannian manifolds and let $N = M_1 \times M_2$ be their Riemannian product, with product metric $g$. Let $\tau_1$, $\tau_2$, $\tau$ and $\kappa_1$, $\kappa_2$, $\kappa$ be the tension fields and conformality operators on $M_1$, $M_2$ and $N$, respectively. Suppose that $\mathcal{E}_1$ and $\mathcal{E}_2$ are eigenfamilies on $M_1$ and $M_2$ with eigenvalues $\lambda_1$, $\mu_1$ and $\lambda_2$, $\mu_2$, respectively. Then the set
\begin{equation*}
\mathcal{E} = \{ \phi: N \to \cn\ | \ \phi: (p_1,p_2) \mapsto \phi_1(p_1)\cdot \phi_2(p_2), \ \ \phi_1 \in \mathcal{E}_1, \phi_2 \in \mathcal{E}_2\}
\end{equation*}
is an eigenfamily on the product space $N$, with
      \begin{equation*}
          \tau(\phi) = (\lambda_1 + \lambda_2)\cdot \phi \quad \mathrm{and} \quad \kappa(\phi,\psi) = (\mu_1 + \mu_2) \cdot \phi\psi
      \end{equation*}
  for all $\phi,\psi \in \mathcal{E}$.
  \end{theorem}

\begin{proof}
Recall that at each point $(p_1,p_2)$ of $M$ we have the orthogonal decomposition
$$
T_{(p_1,p_2)}N = T_{p_1} M_1 \oplus T_{p_2}M_2.
$$
Take $p = (p_1,p_2) \in M$ and let $\{X_1,\dots, X_{m_1}, Y_1, \dots, Y_{m_2}\}$ be a local orthonormal frame for $TM$ on an open subset $U = U_1 \times U_2$ containing $(p_1,p_2)$ such that $\{X_1,\dots, X_{m_1}\}$ and $\{Y_1,\dots, Y_{m_2}\}$ are local orthonormal frames for $TM_1$ and $TM_2$, respectively. 
Employing the above formulae we can now determine the tension field 
\begin{eqnarray*}
\tau(\phi)(p)
&=&\sum_{k=1}^{m_1} \big( X_{k}^2(\phi)(p) -\big(\nab{X_k}{X_k}\big)(\phi)(p)\big)\\ 
& &\qquad + \sum_{l=1}^{m_2} \big(Y_{l}^2(\phi)(p) - \big(\nab{Y_{l}}{Y_{l}}\big)(\phi)(p)\big)\\
&=&\sum_{k=1}^{m_1}\big(X_{k}^2(\phi_1)(p_1)\cdot \phi_2(p_2) - \big(\nab{X_k}{X_k}\big)(\phi_1)(p_1)\cdot \phi_2(p_2)\big)\\
& &+\sum_{l=1}^{m_2} \big(\phi_1(p_1)\cdot Y_{l}^2(\phi_2)(p_2) - \phi_1(p_1)\cdot \big(\nab{Y_{l}}{Y_{l}}\big)(\phi_2)(p_2)\big)\\
&=& \tau_1(\phi_1)(p_1) \cdot \phi_2(p_2) + \phi_1(p_1) \cdot \tau_2(\phi_2)(p_2)\\
&=& \lambda_1 \cdot \phi_1(p_1)\phi_2(p_2) + \lambda_2 \cdot \phi_1(p_1)\phi_2(p_2)\\
&=& (\lambda_1 + \lambda_2)\cdot \psi(p).
\end{eqnarray*}
For the conformality operator $\kappa$ we now have 
\begin{eqnarray*}
\kappa(\phi,\phi)(p)
&=& \sum_{k=1}^{m_1}X_k(\phi)(p)^2 
+ \sum_{l=1}^{m_2}Y_l(\phi)(p)^2\\
&=&\phi_2(p_2)^2\cdot\sum_{k=1}^{m_1}X_{k}(\phi_1)(p_1)^2
 + \phi_1(p_1)^2\cdot\sum_{l=1}^{m_2}Y_{l}(\phi_2)(p_2)^2\\
&=& \cdot \phi_2(p_2)^2\cdot \kappa_1(\phi_1,\phi_1)(p_1) + \phi_1(p_1)^2\cdot \kappa_{2}(\phi_2,\phi_2)(p_2)\\
&=& (\mu_1 + \mu_2)\cdot \phi(p)^2.
\end{eqnarray*}
The general formula for the conformality operator $\kappa$ follows from its symmetry.  
Since the point $p$ was chosen arbitrarily we have completed the proof.
\end{proof}

\section{Acknowledgements}

The authors would like to thank Thomas Jack Munn for useful discussions on this work. 
The second author gratefully acknowledges the support of the Austrian Science Fund (FWF) through the project "The Standard Model as a Geometric Variational Problem"  (DOI: 10.55776/P36862).


\begin{thebibliography}{99}\label{biblio}

\bibitem{Bai-Eel}
P. Baird, J. Eells,
{\it A conservation law for harmonic maps},
Geometry Symposium Utrecht 1980, Lecture Notes in Mathematics {\bf 894}, 1-25, Springer (1981).

\bibitem{Bai-Woo-book}
P. Baird, J.C. Wood, 
{\it Harmonic Morphisms Between Riemannian Manifolds}, 
The London Mathematical Society Monographs {\bf 29}, 
Oxford University Press (2003).


\bibitem{Che-Ebi}
J. Cheeger, D. G. Ebin, 
{\it Comparison Theorems in Riemannian Geometry}, 
North-Holland Mathematical Library {\bf 9} (1975).

\bibitem{Eel-Lem}
J.Eells, L. Lemaire,
{\it Selected topics in harmonic maps}
CBMS Regional Conf. Ser. in Math. {\bf50}, AMS (1983).


\bibitem{Eel-Sam}
J.Eells, J.H.Sampson, 
{\it Harmonic mappings of Riemannian manifolds},
Amer. J. Math. {\bf 86} (1964), 109-160.

\bibitem{Fug}
B. Fuglede, {\it Harmonic morphisms between Riemannian manifolds}, Ann. Inst. Fourier {\bf 28} (1978), 107-144

\bibitem{Gha-Gud-4}
E. Ghandour, S. Gudmundsson,
{\it Explicit $p$-harmonic functions on the real Grassmannians}, 
Adv. Geom. {\bf 23} (2023), 315–321.


\bibitem{Gha-Gud-5}
E. Ghandour, S. Gudmundsson,
{\it Explicit harmonic morphisms and $p$-harmonic functions from the complex and quaternionic Grassmannians}, 
Ann. Global Anal. Geom. {\bf 64} (2023), 15.


\bibitem{Gud-bib}
S. Gudmundsson,
{\it The Bibliography of Harmonic Morphisms},
{\tt www.matematik.lu.se/ 
matematiklu/personal/sigma/harmonic/bibliography.html}

\bibitem{Gud-Riemann},
S. Gudmundsson,
{\it An Introduction to Riemannian Geometry},\\
{\tt www.matematik.lu.se/matematiklu/personal/sigma/Riemann.pdf}


\bibitem{Gud-Mun-1}
S. Gudmundsson, T. J. Munn,
{\it Minimal submanifolds via complex-valued eigenfunctions},
J. Geom. Anal. {\bf 34} (2024), 190.


\bibitem{Gud-Sak-1}
S. Gudmundsson, A. Sakovich, 
{\it Harmonic morphisms from the classical compact semisimple Lie groups}, 
Ann. Global. Anal. Geom. {\bf 33} (2008) 343-356.


\bibitem{Gud-Sif-Sob-2}
S. Gudmundsson, A. Siffert, M. Sobak,
{\it Explicit proper $p$-harmonic functions on the Riemannian symmetric spaces $\SU n/\SO n$, $\Sp n/\U n$, $\SO {2n}/\U n$, $\SU{2n}/\Sp n$},
J. Geom. Anal. {\bf 32} (2022), 147.

\bibitem{Gud-Sob-1}
S. Gudmundsson, M. Sobak,
{\it Proper $r$-harmonic functions from Riemannian manifolds},
Ann. Global Anal. Geom. {\bf 57} (2020), 217-223.

\bibitem{Hel}
S. Helgason, 
{\it Differential Geometry, Lie Groups, and Symmetric Spaces}, 
Graduate Studies in Mathematics {\bf 34}, AMS (2001).

\bibitem{Ish}
T. Ishihara, {\it A mapping of Riemannian manifolds which preserves harmonic functions}, J. Math. Soc. Japan {\bf 7} (1979), 345-370

\bibitem{Lin-MSc}
A. Lindström,
{\it Harmonic Morphisms and $p$-Harmonic Functions on the Classical Compact Symmetric Spaces via the Cartan Embedding},
Master's dissertation, University of Lund (2023).


\end{thebibliography}
\end{document}